\documentclass[11pt]{amsart}
\usepackage{amsmath}
\usepackage{latexsym}
\usepackage{amssymb,amsthm,amsfonts}
\usepackage{graphicx}

\newtheorem{defi}{Definition}[section]
\newtheorem{theorem}[defi]{Theorem}
\newtheorem{lemma}[defi]{Lemma}

\newtheorem{proposition}[defi]{Proposition}
\newtheorem{remark}[defi]{Remark}

\def\R{\mathbb{R}}
\def\Z{\mathbb{Z}}
\def\N{\mathbb{N}}
\def\C{\mathbb{C}}
\def\a{\alpha}

\def\e{\varepsilon}

\def\la{\lambda}

\def\O{\Omega}

\def\intr{\int_{\R^3}}
\def\dis{\displaystyle}
\begin{document}
\title[On the Schr{\"o}dinger-Poisson-Slater system]
{On the Schr{\"o}dinger-Poisson-Slater system: behavior of minimizers,
radial and nonradial cases}

\author{David Ruiz}

\address{Dpto. An\'alisis Matem\'atico, Universidad de Granada, 18071 Granada, Spain}

\thanks{The author has been supported
by the Spanish Ministry of Science and Technology under Grant
MTM2005-01331 and by J. Andaluc\'{\i}a (FQM 116).}

\email{daruiz@ugr.es}


\subjclass[2000]{}

\maketitle

\begin{abstract}
This paper is motivated by the study of a version of the so-called
Schr{\"o}dinger-Poisson-Slater problem:
$$ - \Delta u + \omega u + \lambda \left ( u^2 \star \frac{1}{|x|} \right
) u=|u|^{p-2}u,$$ where $u \in H^1(\R^3)$. We are concerned mostly
with $p \in (2,3)$. The behavior of radial minimizers motivates
the study of the static case $\omega=0$. Among other things, we
obtain a general lower bound for the Coulomb energy, that could be
useful in other frameworks. The radial and nonradial cases turn
out to yield essentially different situations.

\end{abstract}

\section{Introduction}

Our starting point is the system of Hartree-Fock equations:
\begin{equation} \label{hf} - \Delta \psi_k + (V(x)-E_k) \psi_k + \psi_k(x) \int_{\R^3}
\frac{|\rho(y)|^2}{|x-y|}\, dy - \sum_{j=1}^N \psi_j(x)
\int_{\R^3} \frac{\overline{\psi_j(y)}\psi_k(y)}{|x-y|}\,
dy=0,\end{equation} where $\psi_k: \R^3 \to \C$ form an orthogonal
set in $H^1$, $\rho= \frac{1}{N} \sum_{j=1}^N |\psi_j|^2$, $V(x)$
is an exterior potential and $E_k \in \R$. This system appeared in
Quantum Mechanics in the study of a system of $N$ particles. With
respect to the Hartree equations, it has the advantage of being
consistent with the Pauli exclusion principle.

In \eqref{hf}, the last term is usually called the exchange term,
and is the most difficult term to be treated. A very simple
approximation of this term was given by Slater \cite{slater} in
the form:
$$ \sum_{j=1}^N \psi_j \int_{\R^3}
\frac{\overline{\psi_j(y)}\psi_k(y)}{|x-y|}\, dy \sim C_s
\rho^{1/3}\psi_k,
$$
where $C_s$ is a positive constant.

By a mean field approximation, the local density $\rho$ can be
estimated as $\rho = |u|^2$, where $u$ is a solution of the
problem:
$$-\Delta u(x) + V(x) u(x) + B u(x) \int_{\R^3}
\frac{|u(y)|^2}{|x-y|}\, dy =C |u(x)|^{2/3}u(x).
$$
This system receives the name of Schr{\"o}dinger-Poisson-Slater
system (see \cite{bfortunato, boka, bls, mauser}).

\medskip In this paper we are interested in the following version
of the Schr{\"o}dinger-Poisson-Slater problem:

\begin{equation}\label{eq11} - \Delta u + u + \lambda \left ( u^2 \star \frac{1}{|x|} \right )
u=|u|^{p-2}u, \end{equation} where $\lambda>0$. We are concerned
with the case $p \in (2,3)$, and we mainly consider positive
solutions. The case $p\geq 3$ is different and has been studied in
\cite{a-ruiz, kikuchi, JFA}.

In recent years problem \eqref{eq11} has been object of intensive
research, see \cite{a-ruiz, mugnai, mugnai2, daprile2, daprile3,
ianni, kikuchi, kikutesis, pisani, M3AS, JFA, oscar, zhou}. We
point out that \eqref{eq11} presents a combination of repulsive
forces (given by the nonlocal term) and attractive forces (given
by the power term). As we shall see, the interaction between them
gives rise to non expected situations.

The associated energy functional is $I_{\la}: H^1(\R^3) \to \R$,
$$ I_{\la}(u)= \frac 1 2 \int_{\R^3} \left ( |\nabla u|^2 + u^2 \right ) dx +
\frac{\la}{4} \intr \intr \frac{u^2(x) u^2(y)}{|x-y|}\, dx \, dy -
\frac{1}{p} \intr |u|^p\,dx. $$ The original motivation of this
paper is the following. In \cite{JFA} it is shown that
$I_{\la}|_{H_r^1}$ is bounded below for any positive value of
$\lambda$, where $H^1_r$ denotes the Sobolev space of radial
functions. Moreover, when $\lambda$ is small, there exist
nontrivial radial minimizers that blow up as $\lambda \to 0$.

One could ask how is the profile of those solutions as $\lambda
\to 0$. A partial answer is given in \cite{daprile2, M3AS}. In
those papers, by using a perturbation technique, solutions of
\eqref{eq11} with a certain behavior are found (for $\la$ small).
Moreover, those solutions correspond to local minima of
$I_{\la}|_{H^1_r}$ and their energy tend to $-\infty$ as $\lambda
\to 0$, so it is quite reasonable to think that those solutions
correspond to global minima. However, those solutions are provided
only if $p<18/7$. This exponent appears also in more recent work
on concentration on spheres, see \cite{ianni}.

At this point, some natural questions arise: what is the meaning
of the value $p=18/7$? How do minimizers behave if $p \in
(18/7,3)$? Observe that the most important case in applications,
$p=8/3$, belongs to this interval. In this paper we find answers
to both questions.

By making the change of variables $v(x)= \e^{\frac{2}{p-2}}u(\e
x)$, $\e = \lambda^{\frac{p-2}{4(3-p)}}$, we arrive to the
problem:
\begin{equation}\label{eq12} - \Delta v + \e^2 v + \left ( v^2 \star \frac{1}{|x|} \right )
v=|v|^{p-2}v. \end{equation} This motivates the study of the limit
problem:
\begin{equation}\label{eqlimit} - \Delta v + \left ( v^2 \star \frac{1}{|x|} \right )
v=|v|^{p-2}v. \end{equation} Problem \eqref{eqlimit} can be
thought of as a zero mass problem (see \cite{blions}), but under
the action of a nonlocal term. To start with, $H^1(\R^3)$ is not
the right space to study it. It seems quite clear that the right
space should be:
$$ E=E(\R^3)= \{ u \in D^{1,2}(\R^3): \int_{\R^3} \int_{\R^3} \frac{u^2(x) u^2(y)}{|x-y|}\, dx \, dy < +\infty\}.$$
The double integral expression is the so-called Coulomb energy of
the wave, and has been very studied, see for instance \cite{lieb}.
In other words, $E(\R^3)$ is the space of functions in
$D^{1,2}(\R^3)$ such that the Coulomb energy of the charge is
finite. We also denote $E_r=E(\R^3)$ the subspace of radial
functions.

One of the main goals of this paper is the following general
inequality:

\begin{theorem} \label{teolb}
Given $\a >1/2$, there exists $c=c({\a})>0$ such that for any $u:
\R^N \to \R$ measurable function, we have:
\begin{equation} \label{lowerbound}\int_{\R^N} \int_{\R^N}
\frac{u^2(x) u^2(y)}{|x-y|^{N-2}}\, dx \, dy \geq c \left (
\int_{\R^N} \frac{u(x)^2}{|x|^{\frac{N-2}{2}} (1+ \left |\log |x|
\right |)^{\a}}\, dx \right )^2. \end{equation} In particular, $E
\subset L^2(\R^3,\ |x|^{-\frac{1}{2}} (1+ \left |\log |x| \right
|)^{-\a}\, dx)$ continuously.

\end{theorem}

We are not aware of any lower bound for the Coulomb energy in this
fashion. We think that this inequality can be very useful in other
frameworks, such as the Hartree equation or the Thomas-Fermi-Von
Weizs{\"a}cker model (see \cite{catto1, catto2}). We also show that
Theorem \ref{teolb} is ``almost sharp": in the right term, the
exponent $\frac{N-2}{2}$ is optimal and a logarithmic factor is
needed, see Remark \ref{remark}.

By combining inequality \eqref{lowerbound} with the results of
\cite{swwillem, swwillem2}, we obtain the following result, that
shows the significance of the exponent $18/7$ in the radial case.

\begin{theorem} \label{teo2} $E_r(\R^3)\subset L^p(\R^3)$
continuously for $p \in (\frac{18}{7}, 6]$, and the inclusion is
compact for $p \in (\frac{18}{7}, 6)$. Moreover, $E_r(\R^3)$ is
not included in $ L^p(\R^3)$ for $p<\frac{18}{7}$ or $p>6$.
\end{theorem}

With this result in hand, we obtain that for $p \in
(\frac{18}{7},6]$, the functional $J: E_r \to \R$,
$$ J(v)= \frac 1 2 \int_{\R^3} |\nabla v|^2  \, dx +
\frac{1}{4} \intr \intr \frac{v^2(x) v^2(y)}{|x-y|}\, dx \, dy -
\frac{1}{p} \intr |v|^p\,dx, $$ is well-defined, $C^1$, and its
critical points correspond to solutions of \eqref{eqlimit}.
Moreover:

\begin{theorem} \label{teo3}For any $p \in (18/7, 3)$, $J$ is coercive and weak lower
semicontinuous. Therefore, it attains its infimum, which is
negative. As a consequence, \eqref{eqlimit} has a positive
solution in $E$.
\end{theorem}

We do not think that the above solution belongs to $L^2(\R^3)$, so
it does not correspond to a real physical situation. However, it
can be used to describe the asymptotic behavior of the minimizers
of $I_{\la}$:

\begin{theorem} \label{teo4} Suppose that $p \in (18/7,3)$ and let $u_{\lambda}$ be a
minimizer of $I_{\la}|_{H^1_r}$. Then, as $\lambda \to 0$,
$$ u_{\lambda} = \e^{-\frac{2}{p-2}} v_{\e}\left (\frac{x}{\e} \right ), $$
where $\e = \lambda^{\frac{p-2}{4(3-p)}}$ and $d(v_{\e},K) \to 0$.
Here $K\subset E$ is the set of minimizers:
$$ K= \{ v \in E:\ J(v)= \min J\},$$
 and $d(v,K)= \inf \{\|v- w\|_E:\ w \in K \}$. In particular,
 given $\la_n \to 0$, we have that $\e_n \to 0$ and $v_{\e_n} \to v$ in $E$ (up to a subsequence)
 where $v$ is a minimizer of $J$.
 \end{theorem}

We point out that, as $\la \to 0$, radial minimizers behave
differently depending on $p$. For $p>18/7$ minimizers tend to
concentrate around zero and blow up in $L^{\infty}$ norm. On the
other hand, for $p<18/7$, it is reasonable to think that the
solutions given in \cite{daprile2, M3AS} are minimizers; those
solutions spread out and are bounded in $L^{\infty}(\R^3)$.

So far, we have always considered the radial case. The last
section of the paper is devoted to investigate the nonradial case.
We point out that the situation in both cases turns out to be very
different.

To start with, we have the following result (to be compared with
Theorem \ref{teo2}):

\begin{theorem} \label{teo5} $E(\R^3) \subset L^p(\R^3)$ for any $p \in
[3,6)$, and the inclusion does not hold for $p<3$.
\end{theorem}

Moreover, from \cite{JFA} we know that if $p \in (2,3)$,
$I_{\la}|_{H^1_r}$ is always bounded below and attains its
infimum. It is also easy to prove that the map $\la \mapsto \inf
I_{\la}|_{H^1_r}$ is continuous and tends to $-\infty$ as $\la \to
0$. However:

\begin{theorem} \label{teo6}
If $p \in (2,3)$, there exists $\lambda_0>0$ such that $\inf
I_{\la}$=0 for $\lambda \geq \lambda_0$ and $\inf I_{\la}=-\infty$
for $\la < \la_0$.
\end{theorem}

As we see, the nonlocal term leads to different situations in the
radial and nonradial cases. In order to study in depth this
phenomenon, we consider the problem in a ball:

\begin{equation}\label{eqbola}  \begin{array}{ll} -\Delta u + u + \lambda \left ( u^2 \star \frac{1}{|x|} \right )
u=|u|^{p-2}u, & \mbox{ in }B(0,R) \\ \  u(x)=0  & \mbox{ in }
\partial B(0,R). \end{array} \end{equation}
In the following theorem we obtain a result of breaking of
symmetry of minimizers:

\begin{theorem} \label{teo7}
Suppose that either:
\begin{enumerate}
\item[(1)] $p \in (2,3)$, $\lambda \in (0,\lambda_0)$ and $R$
large enough,
\end{enumerate}

\medskip \noindent or
\begin{enumerate}
\item[(2)] $p \in (18/7,3)$, $\lambda$ small.
\end{enumerate}
Then, $\inf \, I_{\la}|_{H^1_0(B(0,R))}$ is attained at a
nonradial function.
\end{theorem}
Observe that here the well-known Gidas-Ni-Nirenberg result
(\cite{gnn}) does not hold (because of the nonlocal term).
Nonradial ground states have been found in other frameworks
previously, like in the H{\'e}non equation, see \cite{smets}. However,
observe that in our case it is a free minimizer of the energy
functional, not a minimizer under a certain constraint. In
particular, it is an orbitally stable solution in the sense of
\cite{calions} (see also Remark \ref{yata}).

The paper is organized as follows. In Section 2 we establish some
notations and we present a preliminary study of the space $E$. In
Section 3 we prove Theorem \ref{teolb}, which provides us with a
general inequality that will be used in the following section.
Section 4 is devoted to the radial case; we prove Theorems
\ref{teo2}, \ref{teo3}, and \ref{teo4}. In Section 5 we deal with
the nonradial case, and prove Theorems \ref{teo5}, \ref{teo6} and
\ref{teo7}.

%

\section{Preliminaries}
In this section we establish some notation that will be used
throughout the paper. We also define the space $E$ and study some
basic properties of it.

We will use the following common notations:

\begin{itemize}

\item $C_0^{\infty}(\R^N)$ is the set of $C^{\infty}$ functions
with compact support.

\item $H^1(\R^3)=\{u \in L^2(\R^3):\ |\nabla u| \in L^2(\R^3)\} $
is the usual Sobolev space, and $\| \cdot \|_{H^1}$ denotes its
norm.

\item $D^{1,2}(\R^N)=\{u  \in L^{\frac{2N}{N-2}}(\R^N):\ |\nabla
u| \in L^2(\R^N)\}$, with the usual norm $\|u\|_{D}= \| \nabla
u\|_{L^2}$.

\item Given any $\O \subset \R^N$ a smooth domain, we
denote by $H^1_0(\O)$ as the completion of $C^{\infty}_0(\O)$ with
the $\|\cdot \|_{H^1}$ norm.

\item We write $C_{0,r}^{\infty}(\R^N)$, $H^1_r(\R^N)$,
$D^{1,2}_r(\R^N)$, $H^1_{0,r}(B(0,R))$ to denote the corresponding
subspaces of radial functions.

\end{itemize}

\begin{defi} We define the space $E$:
$$ E=E(\R^N)= \{ u \in
D^{1,2}(\R^N): \int_{\R^N} \int_{\R^N} \frac{u^2(x)
u^2(y)}{|x-y|^{N-2}}\, dx \, dy < +\infty\}$$ That is, $E(\R^N)$
is the space of functions in $D^{1,2}(\R^N)$ such that the Coulomb
energy of the charge is finite. We denote by $E_r=E_r(\R^N)$ to
the subspace of radial functions.
\end{defi}

\bigskip We begin by studying some elementary properties of $E$:

\begin{proposition} \label{app} Let us define, for any $u\in E$,
$$ \| u \|_E = \left ( \int_{\R^N}
|\nabla u(x)|^2\, dx  + \left ( \int_{\R^N} \int_{\R^N}
\frac{u^2(x) u^2(y)}{|x-y|^{N-2}}\, dx \, dy \right ) ^{1/2}
\right )^{1/2}.$$ Then, $\| \cdot \|_E$ is a norm, and $(E, \|
\cdot \|_E)$ is a uniformly convex Banach space. Moreover,
$C_0^{\infty}(\R^N)$ is dense in $E$, and also
$C_{0,r}^{\infty}(\R^N)$ is dense in $E_r$.
\end{proposition}


\begin{proof}

We will need the following general result, that must be
well-known. Its proof is elementary and will be skipped.


\begin{lemma} \label{lemmino}
Let $X$ be a vectorial space, $p$, $q$ two uniformly convex norms
on $X$, and $\|x\|= \sqrt{p(x)^2+q(x)^2}$. Then $\|\cdot\|$ is
also a uniformly convex norm.
\end{lemma}

%

\bigskip

As usually, (see \cite{catto1, catto2, lieb}), let us define
$$D(f,g)= \int_{\R^N}
\int_{\R^N} \frac{f(x) g(y)}{|x-y|^{N-2}}\, dx \, dy.$$ The
inequality $D(f,g)^2 \leq D(f,f) D(g,g)$ is well-known, see
\cite{lieb}, page 214.

We now show that $\| \cdot \|_E$ is a uniformly convex norm. By
Lemma \ref{lemmino}, it suffices to deal with
$T(u)=D(u^2,u^2)^{1/4}$. First, we show that it satisfies the
triangular inequality. Actually,
$$ D((u+v)^2,(u+v)^2)= D(u^2,u^2) + D(v^2,v^2) + 4 D(u^2,uv)+ 4D(v^2,uv) + 4D(uv,uv) + 2 D(u^2,v^2).$$
We now estimate:
$$ D(u^2,v^2) \leq \sqrt{D(u^2,u^2)D(v^2,v^2)}. $$
In the next computation, we just use H\"{o}lder inequality:
$$ D(uv,uv)= \int_{\R^N}
\int_{\R^N} \frac{u(x)v(x)u(y)v(y)}{|x-y|^{N-2}}\, dx \, dy \leq
\sqrt{D(u^2,u^2)D(v^2,v^2)} .$$ Moreover:
$$ D(u^2, uv) \leq (D(u^2,u^2) D(uv, uv))^{1/2} \leq \left (
D(u^2,u^2) \sqrt{D(u^2,u^2)D(v^2,v^2)}\right )^{1/2}.$$ An
analogous estimate works for $D(v^2,uv)$. Putting all estimates
together, we obtain $T(u+v) \leq T(u)+T(v)$.

With respect to uniform convexity, we can argue as before to
obtain the following inequality: for any $u$, $v\in E$,
$$T\left( \frac{u+v}{2}\right )^4 +T\left( \frac{u-v}{2}\right )^4 \leq \frac{T(u)^4 + T(v)^4}{2}.$$
This readily implies that $T$ is uniformly convex.

In order to show that $E$ is a Banach space take a Cauchy sequence
$u_n$ in $E$. Clearly, $u_n$ is a Cauchy sequence in $D^{1,2}$; we
now show that $\frac{u_n(x)u_n(y)}{|x-y|^{\frac{N-2}{2}}}$ is also
a Cauchy sequence in $L^2(\R^{2N})$. Indeed:
$$ \int_{\R^N} \int_{\R^N} \frac{(u(x)u(y)-v(x)v(y))^2}{|x-y|^{N-2}}\, dx \, dy = $$
$$  \int_{\R^N} \int_{\R^N} \frac{(u(x)u(y)-u(x)v(y)+u(x)v(y)-v(x)v(y))^2}{|x-y|^{N-2}}\, dx \, dy \leq  $$
$$2 \int_{\R^N} \int_{\R^N}
\frac{u(x)^2(u(y)-v(y))^2 + v(y)^2(u(x)-v(x))^2}{|x-y|^{N-2}}\, dx
\, dy =$$
$$ 2 \int_{\R^N} \int_{\R^N}
\frac{(u(x)^2+v(x)^2)(u(y)-v(y))^2}{|x-y|^{N-2}}\, dx \, dy =2 D(
u^2+v^2,(u-v)^2) \leq $$$$ 2 \sqrt{
D(u^2+v^2,u^2+v^2)D((u-v)^2,(u-v)^2)}.$$

So, $u_n \to u$ in $D^{1,2}$ and
$\frac{u_n(x)u_n(y)}{|x-y|^{\frac{N-2}{2}}} \to \psi(x,y)$ in
$L^2$. Passing to a subsequence and by uniqueness of pointwise
convergence we conclude that $\frac{u(x)u(y)}{\
|x-y|^{\frac{N-2}{2}}} = \psi(x,y)$.

Finally, observe that $C_0^{\infty}(\R^N)$ is dense in
$H^1(\R^N)$, and hence $\overline{C_0^{\infty}(\R^N)}^{\, E}
\supset H^1(\R^N)$. So, it suffices to show that $H^1(\R^N)$ is
dense. Take $u \in E$, and choose $\xi \in
C^{\infty}_{0,r}(\R^N)$, $\xi(x)=1$ for $x \in B(0,1)$. It follows
easily that $u_n(x)= \xi(\frac{x}{n})u(x)$ belongs to $H^1(\R^N)$
and $u_n \to u$ in $E$.

Analogously we can argue for the radial case.

\end{proof}

Let us define $\phi_u = \frac{1}{4 \pi |x|^{N-2}} \star u^2$; then,
$u \in E$ if and only if both $u$ and $\phi_u$ belong to
$D^{1,2}(\R^N)$. In such case, $-\Delta \phi_u = u^2$ in a weak
sense, and
$$ \int_{\R^N} |\nabla \phi_u(x)|^2\, dx = \int_{\R^N} \phi_u(x) u(x)^2\, dx = \int_{\R^N}
\int_{\R^N} \frac{u^2(x) u^2(y)}{4 \pi |x-y|^{N-2}}\, dx \, dy.$$

\begin{proposition} Given a sequence $\{u_n\}$ in $E$, $u_n
\rightharpoonup u$ in $E$ if and only if $u_n \rightharpoonup u$
in $D^{1,2}$ and $\int_{\R^N} \int_{\R^N} \frac{u_n^2(x)
u_n^2(y)}{ |x-y|^{N-2}}\, dx \, dy$ is bounded. In such case,
$\phi_n \rightharpoonup  \phi_u$ in $D^{1,2}$, where $\phi_n =
\phi_{u_n}$.
\end{proposition}

\begin{proof}

Clearly, the implication to the right is obvious. Suppose now that
$u_n\rightharpoonup u$ in $D^{1,2}(\R^N)$ and $\int_{\R^N}
\int_{\R^N} \frac{u_n^2(x) u_n^2(y)}{|x-y|^{N-2}}\, dx \, dy$ is
bounded. In particular, $\|u_n\|_E$ is bounded.

Suppose, reasoning by contradiction, that $u_n$ does not converge
weakly to $u$ in $E$. So, there exists a neighborhood of $u$ in
the weak topology and a subsequence (still denoted by $u_n$) such
that $u_n \notin V$. Being $E$ uniformly convex, it is reflexive
(see \cite{brezis}, Theorem III.29), and hence $u_n
\rightharpoonup v $ in $E$ (up to another subsequence) for some $v
\in E$. But this implies that $u_n\rightharpoonup v$ in
$D^{1,2}(\R^N)$, and by uniqueness, $v=u$. This contradicts $u_n
\notin V$.

In order to prove that $\phi_n \rightharpoonup \phi_u$ in
$D^{1,2}$, observe that $\phi_n$ is bounded in $D^{1,2}(\R^N)$,
and therefore $\phi_n \rightharpoonup \phi$ in $D^{1,2}$ for some
$\phi$ (up to a subsequence but, again, this suffices).

Take any $\rho \in C^{\infty}_0(\R^N)$ and compute:
$$ \int_{\R^N} \nabla \phi \cdot \nabla \rho \leftarrow
\int_{\R^N} \nabla \phi_n \cdot \nabla \rho = \int_{\R^N} u_n^2 \rho
\to \int_{\R^N} u^2 \rho.$$ The last convergence follows from the
fact that $(u_n)|_K \to u|_K$ strongly in $L^2$ for any compact set
$K$. Hence, $-\Delta \phi = u^2$, that is, $\phi=\phi_u$.
\end{proof}

\section{A lower bound of the Coulomb energy}

In this section we study some bounds of the Coulomb energy, that
will be of use later on. In particular, we prove Theorem
\ref{teolb}, in which a lower bound of the Coulomb energy is
given.

\medskip

First, by using Hardy-Littlewood-Sobolev inequality (see
\cite{lieb}, page 98), we have the following bound on the Coulomb
energy:
\begin{equation} \label{desig} \int_{\R^N} \int_{\R^N}
\frac{u^2(x) u^2(y)}{|x-y|^{N-2}}\, dx \, dy \leq C \|
u\|_{L^{\frac{4N}{N+2}}}^4.\end{equation} In particular, we have
that $D^{1,2}(\R^N) \cap L^{\frac{4N}{N+2}}(\R^N)\subset E(\R^N) $
continuously.

Let us consider now $E_r(\R^N)$ the subspace of $E$ of radially
symmetric functions. In such case, also $\phi_u$ is a radial
function. For this subspace, we obtain another upper bound:

$$ \| \phi_u\|_D^2 =\int_{\R^N} |\nabla \phi_u(x)|^2\, dx =
\int_{\R^N} \phi_u(x) u(x)^2\, dx \leq C \| \phi_u\|_D \int_{\R^N}
u(x)^2|x|^{-\frac{N-2}{2}}\, dx \Rightarrow $$ \begin{equation}
\label{peso} \int_{\R^N} \int_{\R^N} \frac{u^2(x)
u^2(y)}{|x-y|^{N-2}}\, dx \, dy  \leq C \left ( \int_{\R^N}
u(x)^2|x|^{-\frac{N-2}{2}}\, dx \right )^2.
\end{equation}
In the estimates above we have used the point-wise estimate:
$$v(x) \leq C \|v\|_D
|x|^{-\frac{N-2}{2}} \ \forall \ x \neq 0,$$ for some $C>0$ and
for every $v \in D^{1,2}_r(\R^N)$. This estimate appears in
\cite{blions} (page 340) for $|x|>1$, but a rescaling argument
implies its validity for any $x\neq 0$. Therefore, we have that
$D_r^{1,2}(\R^N) \cap L^2(\R^N, |x|^{-\frac{N-2}{2}}\, dx) \subset
E_r$.

We point out that estimate \eqref{peso} does not hold if $u$ is
not radial, as one can easily check by making use of translations.

Both \eqref{desig} and \eqref{peso} are upper bounds of the
Coulomb energy, which provide us with sufficient conditions for
$u$ to belong to $E$. In this section we prove Theorem \ref{teolb}
which provides us with a \emph{necessary condition}, that is, a
\emph{lower bound}. As far as we know, no lower bound of the
Coulomb energy has been given in the literature so far.

The main result of this section is the following:

\begin{theorem} \label{teo1}
Let $q>0$, $\a
>1/2$. Then there exists $c=c(q,\a)>0$ such that for any $f: \R^N \to [0,+\infty)$ measurable, we have:
$$ \int_{\R^N} \int_{\R^N} \frac{f(x)
f(y)}{|x-y|^{q}}\, dx \, dy \geq c \left ( \int_{\R^N}
\frac{f(x)}{|x|^{\frac{q}{2}} (1+ \left |\log |x| \right
|)^{\a}}\, dx \right )^2.$$
\end{theorem}

Observe that Theorem \ref{teolb} follows trivially from the above
result.

\begin{proof}

We begin by estimating:

$$ \int_{\R^N} \int_{\R^N} \frac{f(x)
f(y)}{|x-y|^{q}}\, dx \, dy \geq  \int_{\R^N} \int_{|y|<2 |x|<4
|y|} \frac{f(x) f(y)}{|x|^{q/2}|y|^{q/2}}
\frac{|x|^{q/2}|y|^{q/2}}{|x-y|^{q}}\, dx \, dy \geq $$$$ c(q)
\int_{\R^N} \int_{|y|<2|x|<4 |y|} \frac{f(x)
f(y)}{|x|^{q/2}|y|^{q/2}} \, dx \, dy =
$$$$c(q) \int_{s=0}^{+\infty} \int_{r=s/2}^{2s} \left (
\frac{1}{r^{q/2}} \int_{|x|=r} f(x)\, d\sigma_x \right ) \left
(\frac{1}{s^{q/2}} \int_{|y|=s} f(y)\, d\sigma_y \right )\ dr \,
ds.
 $$

The rest of the proof is based on the following lemma:

\begin{lemma} \label{lem0} Let $\alpha > 1/2$; then, there exists $c=c(\a)>0$ such
that for any $h:(0,+\infty) \to (0,+\infty)$ measurable function,
there holds:
$$ \int_{s=0}^{+\infty} \int_{r=s/2}^{2s} h(r)(1+|\log r|)^{\a} h(s)(1+|\log s|)^{\a} \, dr \, ds
\geq c \left (\int_0^{+\infty} h(r)\, dr \right )^2.
$$
\end{lemma}

Indeed, we can define: $$h(r)= \frac{1}{r^{q/2}(1+|\log r|)^{\a} }
\int_{|x|=r} f(x)\, d\sigma_x.$$ We now apply the previous lemma
and finish the proof of Theorem \ref{teo1}.

\medskip

In order to prove the lemma, the following inequality will be of
use: for any two sequences of nonnegative numbers $\{a_n\}$,
$\{b_n\}$ ($n \in \Z$), we have:
$$ \left ( \sum_{n=-\infty}^{+\infty} a_n \right)^2 \leq
\left ( \sum_{n=-\infty}^{+\infty} \frac{1}{b_n} \right ) \left (
\sum_{n=-\infty}^{+\infty} b_n a_n^2. \right )
$$
This inequality is quite well-known, but we show briefly the proof
for convenience of the reader. Below we use the inequality $ab
\leq \frac 1 2 (\gamma a^2 + \gamma^{-1} b^2)$:
$$ \left( \sum_{|n|\leq K} a_n \right)^2 = \sum_{|n| \leq K}
\sum_{|m| \leq K} a_n a_m \leq \frac 1 2 \sum_{|n| \leq K}
\sum_{|m| \leq K} \frac{b_n}{b_m} a_n^2 + \frac{b_m}{b_n} a_m^2 =
$$$$ \sum_{|n| \leq K} \sum_{|m| \leq K} \frac{b_n}{b_m}
a_n^2 =  \left ( \sum_{|n| \leq K} \frac{1}{b_n} \right ) \left
(\sum_{|n| \leq K} b_n a_n^2 \right ).
$$ Take $K \to +\infty$ and we are done.

Let us take $b_n = (1+|n|)^{2\a}$ and $a_n = \dis
\int_{2^{n}}^{2^{n+1}} h(r)\, dr$. Then:
$$ \left (\int_0^{+\infty} h(r)\, dr \right )^2 = \left( \sum_{n=-\infty}^{+\infty} a_n \right)^2 \leq
C(\a) \sum_{n=-\infty}^{+\infty} (1+|n|)^{2\a} a_n^2=$$ $$ C(\a)
\sum_{n=-\infty}^{+\infty}  \left ( (1+|n|)^{\a}
\int_{2^{n}}^{2^{n+1}} h(r) \, dr \right )^2 = $$$$ C(\a)
\sum_{n=-\infty}^{+\infty}  \left ( (1+|n|)^{\a}
\int_{2^{n}}^{2^{n+1}} h(r) \, dr \right ) \left ( (1+|n|)^{\a}
\int_{2^{n}}^{2^{n+1}} h(s) \, ds \right ).
$$

Now, for $n \geq 0$, we have that $r>2^{n}$, hence $n< \log_2 r=
\log_2(e) \log r$, so $1+n \leq C (1+ \log r)$.

\medskip Moreover, if $n<0$, we estimate $2^{n+1}>r \Rightarrow 0
\geq n+1
> \log_2 r =\log_2(e) \log r$. We take absolute values and obtain
that $|n|-1 < \log_2 e |\log r|$, so $1+|n| \leq C (1+ |\log r|)$.
Analogously, $1+|n| \leq C (1+ |\log s|)$ in any integral term
above.

\medskip

Hence:
$$\sum_{n=-\infty}^{+\infty}  \left ( (1+|n|)^{\a}
\int_{2^{n}}^{2^{n+1}} h(r) \, dr \right ) \left ( (1+|n|)^{\a}
\int_{2^{n}}^{2^{n+1}} h(s) \, ds \right ) \leq $$$$ C
\sum_{n=-\infty}^{+\infty} \int_{2^{n}}^{2^{n+1}}
\int_{2^{n}}^{2^{n+1}} h(r)(1+ |\log r|)^{\a} \ h(s)(1+ |\log
s|)^{\a} \, dr \, ds \leq $$$$ C \int_{s=0}^{+\infty}
\int_{s/2}^{2s} h(r)(1+ |\log r|)^{\a} h(s)(1+ |\log s|)^{\a} \,
dr ds.$$

\end{proof}

\begin{remark} \label{remark} Observe that in \eqref{lowerbound} the exponent
$\frac{N-2}{2}$ is the same as in inequality \eqref{peso} for
radial functions. With respect to the logarithmic term, we do not
know whether inequality \eqref{lowerbound} holds for some smaller
value of $\a$ or not. However, we show now that \eqref{lowerbound}
does not hold if $\alpha < \frac{N-2}{2N}$. Indeed, by combining
\eqref{desig} with the thesis of Theorem \ref{teolb}, we conclude
the inclusion:
\begin{equation} \label{eoeoeo} L^{\frac{4N}{N+2}}(\R^N) \subset
L^2(\R^N,\ |x|^{-\frac{N-2}{2}} (1+ \left |\log |x| \right
|)^{-\a}\, dx).\end{equation} Actually one can check that
inclusion directly by using H{\"o}lder inequality if $\alpha >
\frac{N-2}{2N}$. But \eqref{eoeoeo} is false for $\alpha<
\frac{N-2}{2N}$, as one can easily check by using the function:
$$f(x)=
\frac{1}{|x|^{\frac{N+2}{4}} (1+\left | \log |x|\right
|)^{\beta}},$$ where $\frac{N+2}{4N}<\beta \leq
\frac{1-\alpha}{2}$. Hence, the thesis of Theorem \ref{teolb} does
not hold with $\alpha< \frac{N-2}{2N}$.

Observe that the previous argument implies that Lemma \ref{lem0}
is not true for any $\alpha <1/2$. Indeed, if $\alpha < 1/2$, we
can always choose $N$ large enough so that $\alpha<
\frac{N-2}{2N}$. Observe now that Lemma \ref{lem0} would imply
\eqref{lowerbound}, which is false by the previous argument.

\end{remark}

\begin{remark} By using translations, we can obtain the
following generalization of Theorem \ref{teo1}:

Let $q>0$, $\a
>1/2$. Then there exists $c=c(q,\a)>0$ such that for any $f: \R^N \to [0,+\infty)$ measurable and any
$z \in \R^N$, we have:
$$ \int_{\R^N} \int_{\R^N} \frac{f(x)
f(y)}{|x-y|^{q}}\, dx \, dy \geq c \left ( \int_{\R^N}
\frac{f(x)}{|x-z|^{\frac{q}{2}} (1+ \left |\log |x-z| \right
|)^{\a}}\, dx \right )^2.$$

\end{remark}

\section{The radial case}

From now on we will restrict ourselves to the case $N=3$, since
this is the most interesting case in applications. In this section
we are concerned with the radial case, and we will prove Theorems
\ref{teo2}, \ref{teo3} and \ref{teo4}.

First of all, if $u$ is a radial function, $\phi_u$ is also
radial, and can be written as (see \cite{daprile2, M3AS}):
$$ \phi_u(r)= \frac{1}{r} \int_0^{+\infty} u^2(s) s min\{r,s\}\, ds. $$
Therefore, $$\| \phi_u \|_D^2 = \intr \phi_u(|x|) u^2(|x|)\, dx =
4\pi  \int_0^{+\infty}  \int_0^{+\infty} u^2(r) u^2(s) r s
min\{r,s\} \, dr \, ds.$$

\begin{proof}[{\bf Proof of Theorem \ref{teo2}}]

\medskip

Take any $\gamma > \frac 1 2$, and define $V(x)=
(1+|x|)^{-\gamma}$. Following \cite{swwillem, swwillem2}, let us
define:
$$H_r^1(\R^3,V):= D_r^{1,2}(\R^3) \cap L^2(\R^3,\ V(x)\, dx).$$
In that space we consider the norm $\|u\|^2_{H_V}= \int_{\R^3}
|\nabla u(x)|^2 + V(x) u^2(x) \, dx$.

By Theorem \ref{teolb} we obtain $E_r(\R^3) \subset H_r^1(\R^3,V)$.
The spaces of radial functions $H_r^1(\R^N,V)$ have been studied in
\cite{swwillem, swwillem2}: there it is proved that
$$H_r^1(\R^3,V) \subset L^p(\R^3) \ \mbox{for } p\in
\left [ \frac{2(4+\gamma)}{4-\gamma},6 \right ]$$ with continuous
inclusion. Observe now that fixed $p > \frac{18}{7}$, we can take
$\gamma
>\frac{1}{2}$ such that $p=\frac{2(4+\gamma)}{4-\gamma}$.

\medskip We can now prove the compact inclusions as usually, combining the
continuous inclusions with asymptotic estimates. Fix $p \in
(\frac{18}{7},6)$ and assume, without loss of generality, that
$u_n \rightharpoonup 0$ in $E$. In particular, $u_n
\rightharpoonup 0$ in $D^{1,2}$ and $u_n \rightharpoonup 0$ in
$L^q$ for every $q \in (\frac{18}{7},6]$.

We have the asymptotic estimate $|u_n(x)| \leq C \|u_n\|_{D}
|x|^{-1/2} \leq C'|x|^{-1/2}$ (actually, a better estimate can be
given, see \cite{swwillem, swwillem2}). Choose $\delta>0$ such
that $p-\delta > \frac{18}{7}$. We have:
$$ \int_{B(0,R)^c} u_n^p= \int_{B(0,R)^c} u_n^{\delta}
u_n^{p-\delta}\leq C R^{-\delta/2} \int_{\R^N} u_n^{p-\delta} \leq
C' R^{-\delta/2}.$$ So, given $\e>0$, we can choose $R>0$ such
that $ \int_{B(0,R)^c} u_n^p < \e$ for every $n \in \N$. On the
other hand, $u_n \to 0$ in $L^p(B(0,R))$, and this finishes the
proof of compactness.

\medskip By using dilatations, it is easy to see that $E$ is not included in $L^{p}(\R^3)$ for $p>6$.
In order to deal with the case $p < 18/7$, we prove the following

{\bf Claim:} Given $p<18/7$, $M>0$, $T>0$, there exists a function
 $u \in E$ with compact support such that $\| u\|_E \leq 1$ , $u=0$ in $B(0,M)$, $\intr |u|^{p}
>T$.

\medskip Indeed, for any $\e \in (0,1)$ define $R=\e^{-8/7}$, $S=\e^{-2/7}$
and:
$$ u_{\e}(r)= \left \{ \begin{array}{ll} 0  & \mbox{ if } |r-R| \geq
S, \\ \e \ \frac{S-|r-R|}{S} & \mbox{ if } |r-R| <S.
\end{array} \right. $$

We compute the norm of $u_{\e}$ in $E$:

$$ \int_0^{+\infty} u'(r)^2 r^2\, dr \leq   \int_{R-S}^{R+S}
\frac{\e^2}{S^2} (R+S)^2\, dr = 2 \frac{\e^2}{S} (2R)^2 \leq 8,$$

$$ \int_0^{+\infty}  \int_0^{+\infty} u^2(r) u^2(s) r s min\{r,s\} \,
dr \, ds \leq $$$$\int_{R-S}^{R+S} \int_{R-S}^{R+S} \e^4 (R+S)^3\,
dr \, ds \leq 4 S^2 \e^4  (2R)^3 \leq 32.$$

Moreover:
$$ \int_0^{+\infty} u(r)^p r^2 \, dr \geq \int_{R-S/2}^{R+S/2}
\left (\frac{\e}{2}\right )^p (R-S/2)^2 \, dr \geq $$$$S \left
(\frac{\e}{2}\right )^p (R/2)^2= \frac{1}{2^{p+2}} \e^p S R^2 =
\frac{1}{2^{p+2}} \e ^{p-\frac{18}{7}}.$$

Hence, the claim follows by taking $\e$ small enough (and dividing
by a convenient constant).

\bigskip Observe that the above claim readily implies that there
is no continuous inclusion $E \subset L^p(\R^3)$ for $p <
\frac{18}{7}$. We now show briefly that there is no inclusion at
all, continuous or not. By the above claim, we can construct a
sequence $u_n$ such that $\|u_n\|_E \leq 2^{-n}$, $\intr
|u_n(x)|^p\, dx =1$ and with disjoint support. Observe that $v=
\sum_{n=1}^{+\infty} u_n \in E$ and $\|v\|_E \leq 1$. Moreover,
since $u_n$ have disjoint support, we have:
$$ \intr |v(x)|^{p} \, dx = \sum_{n=1}^{+\infty} \intr |u_n(x)|^{p} \, dx = +\infty. $$

\end{proof}

\begin{remark} We conjecture that $E$ is not included in
$L^{18/7}(\R^3)$. \end{remark}

\begin{proof}[{\bf Proof of Theorem \ref{teo3}}]

\medskip

By Theorem \ref{teo2}, $J$ is well-defined, and it can be checked
that $J$ is $C^1$ and that:
$$ J'(u)(v)= \intr \nabla u \cdot \nabla v\, dx + \intr \intr \frac{u^2(x) u(y)v(y)}{|x-y|}
\, dx \, dy -\intr |u|^{p-2}uv. $$ Take $u\in E$, and define:
$$ M(u)= \intr |\nabla u|^2 \, dx + \int_{\R^3} \int_{\R^3} \frac{u^2(x)
u^2(y)}{|x-y|}\, dx \, dy. $$

Define $\lambda= M(u)^{-1/3}$, and $v(x)= \lambda^2 u(\lambda x)$.
Observe that: $ M(v)= \lambda^3 M(u) =1$, so, $\|v\|_E \leq
\sqrt{2}$. By Theorem \ref{teo2}, there exists $C>0$ such that
$\intr |v|^p < C$. Now observe that $\intr |u|^p =
\lambda^{3-2p}\intr |v|^p$. Summing up:
$$ J(u) \geq  \frac 1 4 M(u) - \frac{1}{p} \intr |u|^p\, dx \geq
\frac 1 4 M(u) - \frac C p M(u)^{\frac{2p-3}{3}} \geq \frac 1 8
M(u) -C'.
$$
For last inequality just observe that the function $g(s)= \frac s
8 - \frac C p s^{\frac{2p-3}{3}}$ is bounded below for $p<3$.
Hence, $J$ is coercive.

By using the compactness of the inclusion $E \subset L^p(\R^3)$,
it is easy to show that $I$ is weakly lower semicontinuous.

In order to finish the proof of Theorem \ref{teo3} we just need to
show that $\min J <0$. For that, fix $u \in E_r$, and define again
$v_{\la}(x)= \lambda^2 u(\lambda x)$. Hence:
$$ J(v_{\la})= \frac{\la^3}{2} \intr |\nabla u|^2 \, dx + \frac{\la^3}{4} \int_{\R^3} \int_{\R^3} \frac{u^2(x)
u^2(y)}{|x-y|}\, dx \, dy -\frac{\la^{2p-3}}{p} \intr |u|^p\, dx.
$$
So, for $\lambda$ small, $J(v_{\la})$ takes negative values.

\end{proof}

\begin{proof}[{\bf Proof of Theorem \ref{teo4}}]

By making the change of variables $\e =
\lambda^{\frac{p-2}{4(3-p)}}$, $v(x)= \e^{\frac{2}{p-2}}u(\e x)$,
problem \eqref{eq11} is equivalent to:
\begin{equation}\label{eq32} - \Delta v + \e^2 v + \left ( v^2 \star \frac{1}{|x|} \right )
v=|v|^{p-2}v. \end{equation} Let us define the associated
functional $J_{\e}: H^1_r(\R^3) \to \R$,
$$ J_{\e}(v)= \frac 1 2 \int_{\R^3} \left ( |\nabla v|^2 + \e^2 |v|^2 \right ) dx +
\frac{1}{4} \intr \intr \frac{v^2(x) v^2(y)}{|x-y|}\, dx \, dy -
\frac{1}{p} \intr |v|^p\,dx. $$

Observe that $J_{\e}(v)= \e^{\frac{6-p}{p-2}}I_{\la}(u)$.

Take $u_{\la}$ minimizer of $I_{\la}$, and define $v_{\e}$ through
the above change of variables: clearly, $v_{\e}$ is a minimizer
for $J_{\e}$.

\medskip

First, we claim that $\min J_{\e} \to \min J$ as $\e \to 0$.
First, observe that $\min J_{\e} > \min J$. Take any $\delta >0$
and choose $v \in E$ such that $J(v) = \min J$. Since
$C_{0,r}^{\infty}(\R^3)$ is dense in $E_r$ (see Proposition
\ref{app}), we can choose $\rho \in C_{0,r}^{\infty}(\R^3)$ such
that $J(\rho)< \min J + \delta$. Hence, we can choose $\e_0>0$
such that for $\e\in (0,\e_0)$, $J_{\e}(\rho) < \min J + 2
\delta$. Since $\delta$ is arbitrary, we conclude the proof of the
claim.

\medskip

Hence, $\min J \leftarrow \min J_{\e}=J_{\e}(v_{\e}) \geq
J(v_{\e}) \geq \min J$, so $v_{\e}$ is minimizing for $J$.
Suppose, reasoning by contradiction, that there exists $c_0>0$ and
$\e_n \to 0$ such that $d (v_{\e_n},K)>c_0$. Since $J$ is coercive
(Theorem \ref{teo3}), we have that $v_{\e_n}$ is bounded in $E$,
and hence it converges weakly to some $v \in E$. Moreover, $J$ is
weakly lower semicontinuous, so $v \in K$.

Recall that $J(v_{\e_n}) \to J(v)$, and, by Theorem \ref{teo1},
$\intr |v_n|^p \to \intr |v|^p$. From these facts we deduce that
$\|v_{\e_n}\|_E \to \|v\|_E$. Since $E$ is uniformly convex, we
conclude that $v_{\e_n} \to v$ strongly in $E$ (see \cite{brezis},
Proposition III.30). This is a contradiction with $d(v_{\e_n}, K)
>c_0>0$.

\end{proof}

\section{The nonradial case}

In this final section we consider the nonradial case, and we will
prove Theorems \ref{teo5}, \ref{teo6}, \ref{teo7}.

We start with some estimates that will be of use later on. Given
any $u \in C_0^{\infty}(\R^3)$ with support included in $B(0,M)$,
$e \in \R^3$ with $|e|=1$, and $N\in \N$, we define:
\begin{equation} \label{uN} u_N(x)= \sum_{i=1}^N u(x+i N^2e).
\end{equation}

Observe that $u_N$ is a sum of translations of $u$, and if
$N^2>2M$ the summands have disjoint support. In such case we have:

\begin{equation} \label{c1} \intr |\nabla u_N|^2 \, dx = N \intr |\nabla u|^2 \,
dx,\end{equation} \begin{equation} \label{c2} \intr |u_N|^p \, dx
= N \intr |u|^p \, dx,
\end{equation}
$$ \intr \intr \frac{u_N^2(x) u_N^2(y)}{|x-y|}\, dx \, dy =
\sum_{i,j=1}^N \intr \intr \frac{u^2(x+iN^2e)
u^2(y+jN^2e)}{|x-y|}\, dx \, dy = $$$$N \intr \intr \frac{u^2(x)
u^2(y)}{|x-y|}\, dx \, dy + \sum_{i \neq j}^N \intr \intr
\frac{u^2(x+iN^2e) u_N^2(y+jN^2e)}{|x-y|}\, dx \, dy.$$

Now we compute:

$$ \sum_{i \neq j}^N \intr \intr
\frac{u^2(x+iN^2e) u^2(y+jN^2e)}{|x-y|}\, dx \, dy = $$$$\sum_{i
\neq j}^N \int_{B(0,M)} \int_{B(0,M)} \frac{u^2(x) u^2(y)}{|x-y +
(j-i)N^2e|}\, dx \, dy\leq
$$$$\sum_{i \neq j}^N \int_{B(0,M)} \int_{B(0,M)} \frac{u^2(x)
u^2(y)}{N^2-2M}\, dx \, dy =\frac{N^2-N}{N^2-2M} \left (\intr
u^2(x) \, dx \right )^2.$$

So, we get:

\begin{equation} \label{c3} \left |
\intr \intr \frac{u_N^2(x) u_N^2(y)}{|x-y|}\, dx \, dy - N \intr
\intr \frac{u^2(x) u^2(y)}{|x-y|}\, dx \, dy \right | \leq C,
\end{equation}
where $C>0$ is a constant depending on the original function $u$
and is independent of $N$.

\medskip \begin{proof}[{\bf Proof of Theorem \ref{teo5}}]

First, we show that $E \subset L^3(\R^3)$. Actually, this is
well-known from \cite{lions}. Recall that $-\Delta \phi = u^2$,
and $\phi \in D^{1,2}(\R^3)$. By multiplying by $|u|$ and
integrating, we obtain: \begin{equation} \label{lions} \intr |u|^3
 = \intr \langle \nabla \phi , \nabla |u|
\rangle \leq \frac 1 2 \intr |\nabla u|^2 + |\nabla \phi|^2\, dx.
\end{equation}

Recall now that:
$$ \int_{\R^3} |\nabla \phi_u(x)|^2\, dx =  \int_{\R^3}
\int_{\R^3} \frac{u^2(x) u^2(y)}{4 \pi |x-y|}\, dx \, dy.$$

This implies that $E \subset L^3(\R^3)$. Moreover, $E \subset
D^{1,2}(\R^3) \subset L^6(\R^3)$ by Sobolev embedding. By
interpolation we conclude that $E \subset L^p(\R^3)$ for all $p
\in [3,6]$.

Let us now show that $E$ is not included in any $L^p$ space for
$p<3$. Fix $u \in C_0^{\infty}(\R^3)$; for $N \in \N$ define $u_N$
as in \eqref{uN}. Take also $\lambda_N = N^{-1/3}$, and define
$v_N= \lambda_N^2 u_N(\lambda_N x)$. By using \eqref{c1},
\eqref{c2}, \eqref{c3}, we obtain:

$$ \intr |\nabla v_N|^2 \, dx = \lambda_N^3 \intr |\nabla u_N|^2=
C_1.$$
$$ \intr \intr \frac{v_N^2(x) v_N^2(y)}{|x-y|}\, dx \, dy =
\lambda_N^3 \intr \intr \frac{u_N^2(x) u_N^2(y)}{|x-y|}\, dx \,
dy$$$$ \leq \la_N^3 (C_2 N + C_3) \leq C_4.$$

$$ \intr |v_N|^p \, dx = \lambda_N^{2p-3} \intr |u_N|^p \, dx =
C_5\lambda_N^{2p-3}N.$$

Above, $C_k$ are positive constants that depend upon $u$, but are
independent of $N$. So, $\{v_N\}$ is a bounded sequence in $E$
such that $\int |v_N|^p\, dx \to +\infty$ if $p<3$.

This already implies that there is no continuous inclusion from
$E$ in $L^p(\R^3)$. Now, we can argue as in the proof of Theorem
\ref{teo2}, to conclude that $E$ is not included in $L^p$.

\end{proof}

We now turn our attention to Theorem \ref{teo6}. First of all, we
remind a result of \cite{JFA} (Theorem 4.3):

\begin{proposition} \label{ta}For any $\la>0$ and $p\in (2,3)$, $I_{\la}|_{H^1_r(\R^3)}$ is
coercive and weak lower semicontinuous. In particular, it has a
minimum. \end{proposition}

As we see, Theorem \ref{teo6} is in contrast with the above
result.

To start with, let us consider the map $m: [0, +\infty) \to
[-\infty, +\infty)$, $m(\la)= \inf I_{\la}$. It is easy to check
that $m$ is nondecreasing and upper semicontinuous, and
$m(0)=-\infty$.

Moreover, $m_{\la} \leq I_{\la}(0)=0$ for any $\la \in \R$.
Actually, in \cite{JFA} it is proved that $m_{\la}=0$ for $\la$
large and $p \in (2,3)$. Let us reproduce the proof here, for the
sake of completeness:

$$ I_{\la}(u)= \frac 1 2 \int_{\R^3} \left ( |\nabla u|^2 + u^2  \right ) dx +
\frac{\la}{4} \intr \intr \frac{u^2(x) u^2(y)}{|x-y|}\, dx \, dy -
\frac{1}{p} \intr |u|^p\,dx =$$ $$  \frac 1 2 \int_{\R^3} \left
(|\nabla u|^2 + u^2 \right )  dx + \pi \la \intr |\nabla
\phi_u|^2\, dx - \frac{1}{p} \intr |u|^p\,dx. $$

Then, if $\lambda \pi \geq \frac 1 2$, we can use inequality
\eqref{lions} to conclude:

$$ I_{\la}(u) \geq  \int_{\R^3}
\left ( \frac 1 2 u^2 +  |u|^3  - \frac{1}{p}  |u|^p \right )dx
\geq 0.$$

\begin{proof}[{\bf Proof of Theorem \ref{teo6}}]

\medskip With all these preliminaries, Theorem \ref{teo6} follows
from the following

\medskip {\bf Claim: } If $m_{\la}<0$, then $m_{\la}=-\infty$.

\medskip

If $m_{\la}<0$, we can use density of $C^{\infty}_0(\R^3)$ in
$H^1(\R^3)$ to find $u \in C_0^{\infty}(\R^3)$ such that
$I_{\la}(u)<0$. Given such function $u$, define $u_N$ as in the
beginning of the section. By \eqref{c1}, \eqref{c2}, \eqref{c3},
$$ I_{\la}(u_N)\leq  N I_{\la}(u) + C,$$
where $C>0$ is independent of $N$. So, we conclude that $\lim_{N
\to +\infty} I_{\la}(u_N) = -\infty$.

\end{proof}

As we have seen, Theorems \ref{teo5} and \ref{teo6} make clear the
differences between the radial and nonradial cases. This
phenomenon is due to the nonlocal term given by the Coulomb
energy; observe that this term increases when we make Schwartz
rearrangements.

In order to obtain more consequences from this, let us consider
the problem:

\begin{equation}\label{eqbola2}  \begin{array}{ll} -\Delta u + u + \lambda \left ( u^2 \star \frac{1}{|x|} \right )
u=|u|^{p-2}u, & \mbox{ in }B(0,R) \\ \  u(x)=0  & \mbox{ in }
\partial B(0,R). \end{array} \end{equation}

As always, $p\in (2,3)$. The associated energy functional is
nothing but $I_{\la}|_{H_0^1(B(0,R))}$ (in the radial case,
$I_{\la}|_{H_{0,r}^1(B(0,R))}$). Define:

$$ m(R, \la)= \inf  I_{\la}|_{H_0^1(B(0,R))}, \ \
\bar{m}(R, \la)= \inf  I_{\la}|_{H_{0,r}^1(B(0,R))}. $$

\begin{lemma} The infima that define $m$ and $\bar{m}$ are achieved.
\end{lemma}

\begin{proof}

The proof follows from \cite{JFA}. Indeed, as in \eqref{lions},

\begin{equation} \label{eq34} c_{\lambda} \int_{B(0,R)} |u|^3 =
c_{\lambda} \int_{B(0,R)} \langle \nabla \phi , \nabla |u| \rangle
\leq \frac{1}{4} \int_{B(0,R)}  |\nabla u|^2 + \frac{\lambda}{4}
\intr |\nabla \phi|^2\, ,
\end{equation} where $c_{\lambda}= \frac{\sqrt{\lambda}}{2}>0$.
Let us point out that $\phi= \frac{1}{4\pi |x|} \star u^2$, so
$-\Delta \phi = u^2$ in $\R^N$, so $\phi$ is not equal to zero on
the boundary.

We have:
$$I_{\la}(u) \geq \int_{B(0,R)} \left ( \frac{1}{4}  |\nabla u|^2+ \frac{1}{2} u^2 +
c_{\lambda}|u|^3- \frac{1}{p+1}|u|^{p+1}\right) dx \geq
\int_{B(0,R)} \frac{1}{4} |\nabla u|^2\, dx -C. $$ Therefore,
$I_{\la}|_{H_0^1(B(0,R))}$ is coercive. It is not difficult to
check that it is also weak lower semicontinuous, so the existence
of a minimum holds. The same arguments work in the radial case.

%
%
%
%
%

\end{proof}

\begin{proof}[{\bf Proof of Theorem \ref{teo7}}]

$ $

\begin{enumerate}

\item[(1)] Suppose that $p \in (2,3)$, $\la  \in (0,\la_0)$, where
$\la_0$ is given in Theorem \ref{teo6}.

\end{enumerate}

It suffices to show that for $R$ large, $m(R,\la) <
\bar{m}(R,\la)$. Have in mind Theorem \ref{teo6} and Proposition
\ref{ta}, so:
$$ \lim_{R \to +\infty} m(R,\la)= \inf I_{\la} = -\infty,\ \
\lim_{R \to +\infty} \bar{m}(R,\la)= \inf I_{\la}|_{H^1_r(\R^3)}>
-\infty.$$ Then, for $R$ large, $m(R,\la) < \bar{m}(R,\la)$.

\begin{enumerate}

\item[(2)] We now assume $p \in (18/7,3)$ and $R$ fixed.

\end{enumerate}
This case does not follow as above, since $\lim_{\la \to 0}
m(R,\la)= \lim_{\la \to 0} \bar{m}(R,\la)= -\infty$.

We make again the change of variables:
 $v(x)= \e^{\frac{2}{p-2}}u(\e
x)$, $\e = \lambda^{\frac{p-2}{4(3-p)}}$, to arrive to the
problem:

\begin{equation}\label{eqbola3}  \begin{array}{ll} - \Delta v + \e^2 v + \left ( v^2 \star \frac{1}{|x|} \right )
v=|v|^{p-2}v, & \mbox{ in }B(0,R/\e) \\ \  v(x)=0  & \mbox{ in }
\partial B(0,R/\e). \end{array} \end{equation}

As in the proof of Theorem \ref{teo4}, define $$ J_{\e}(v)= \frac
1 2 \int_{\R^3} \left ( |\nabla v|^2 + \e^2 |v|^2 \right )  dx +
\frac{1}{4} \intr \intr \frac{v^2(x) v^2(y)}{|x-y|}\, dx \, dy -
\frac{1}{p} \intr |v|^p\,dx. $$ Define $c(\e)= \inf
J_{\e}|_{H^1_0(B(0,R/\e))}$, $\bar{c}(\e)= \inf
J_{\e}|_{H^1_{0,r}(B(0,R/\e))}$. Since $J_{\e}(v)=
\e^{\frac{6-p}{p-2}}I_{\la}(u)$, we have that:

$$ c(\e)= \e^{\frac{6-p}{p-2}} m(\la,R),\ \ \bar{c}(\e)= \e^{\frac{6-p}{p-2}} \bar{m}(\la,R). $$

Observe now that $\bar{c}(\e) \geq \inf J_0|_E >-\infty$ (Theorem
\ref{teo3}). We now claim that $\lim_{\e \to 0} c(\e) = -\infty$,
and this finishes the proof.

Take $M>0$ arbitrary; by Theorem \ref{teo5}, $E$ is not included
in $L^p(\R^3)$, and there exists $u\in C^{\infty}_0(\R^3)$ such
that $\|u\|_E \leq 1$, $\intr |u|^p >M$. Now choose $\e>0$ small
enough such that:
$$supp \ u \subset
B(0,R/\e) \ \mbox{ and } \ \e^2 \intr |u|^2 \leq 1.$$

In such case, $u \in H^1_0(B(0,R/\e))$ and $c(\e)\leq J_{\e}(u)
\leq 3 -M$.

\end{proof}

\begin{remark} \label{yata} Being $u$ a nonradial minimizer of $I_{\la}|_{H_0^1(B(0,R))}$, we
have a family of minimizers $\{u \circ g: \ g \in O(N)\}$. This
implies that the minimum is degenerate, and hence it does not
satisfy the conditions of \cite{grillakis} for orbital stability.
Let us denote by $K$ the set of minimizers of
$I_{\la}|_{H_0^1(B(0,R))}$; then, $K$ is orbitally stable in the
sense of \cite{calions}.

\end{remark}

\begin{remark}
In the above result we have emphasized the breaking of symmetry of
the minimizer: let us now consider briefly the multiplicity of
positive solutions.

Under the conditions of Theorem \ref{teo7} we actually obtain the
existence of two solutions, a minimizer for $m(R,\la)$ and a
different one for $\bar{m}(R,\la)$. Moreover, $m(R, \la)$ is
negative and by taking $\la$ smaller (if necessary),
$\bar{m}(R,\la)$ is also negative. So, both minima are negative,
and hence they yield two positive nontrivial solutions.

Even more, observe that $0$ is a local minimum of
$I_{\la}|_{H^1(\R^3)}$. It is not difficult to show that the
Palais-Smale condition holds (remember that $I_{\la}$ is
coercive). The well-known mountain-pass theorem (\cite{a-rab})
implies the existence of a third nontrivial solution.

Everything said above can be applied to the functional:
$$ I^+_{\la}(u)= \frac 1 2 \int_{\R^3} \left ( |\nabla u|^2 + u^2 \right ) dx +
\frac{\la}{4} \intr \intr \frac{u^2(x) u^2(y)}{|x-y|}\, dx \, dy -
\frac{1}{p} \intr (u^+)^p\,dx. $$ By the maximum principle, we
obtain the existence of three positive solutions.

\end{remark}

{\bf Aknowledgement: } The author thanks Prof. Ireneo Peral for
many discussions on these problems during a stay in the
Universidad Aut{\'o}noma of Madrid, as well as for the warm
hospitality.


\begin{thebibliography}{99}

%
%

\bibitem{a-rab}{A. Ambrosetti and P. H. Rabinowitz, }{Dual
variational methods in critical point theory and applications,
}{J. Funct. Anal. 14 (1973), 349-381.}

\bibitem{a-ruiz}{A. Ambrosetti and D. Ruiz, }{Multiple bound states for the Schr{\"o}dinger-Poisson problem,
}{to appear in Comm. Contemp. Math.}

\bibitem{bfortunato}{V. Benci and D. Fortunato, }{An
eigenvalue problem for the Schr{\"o}dinger-Maxwell equations, }{Top.
Meth. Nonl. Anal. 11 (1998), 283-293.}

\bibitem{boka}{O. Bokanowski and N.J. Mauser, }{
Local approximation of the Hartree-Fock exchange potential: a
deformation approach, }{M$^3$AS 9 (1999), 941-961.}

\bibitem{bls}{O. Bokanowski, J.L. L{\'o}pez and J. Soler, }
{On an exchange interaction model for the quantum transport; the
Schr\"odinger-Poisson-Slater term, } {M$^3$AS 13 (2003),
1397-1412.}


\bibitem{brezis}{H. Brezis, }{Analyse fonctionelle, Th{\'e}orie et applications, }{Masson Ed., Paris,
1983.}

\bibitem{blions}{H. Beresticky and P. L. Lions, }{Nonlinear Scalar field equations
I: Existence of a ground state, }{Arch. Rat. Mech. Anal 82 (1983),
313-345.}


\bibitem{catto1}{I. Catto, C. Le Bris and P. L. Lions, }{On some periodic
Hartree-type models for crystals, }{Ann. Inst. H. Poincar{\'e} Anal.
Non Lin{\'e}aire 19 (2002), 143-190. }

\bibitem{catto2}{I. Catto, C. Le Bris, and P. L. Lions, }{On the thermodynamic limit
for Hartree-Fock type models, }{Ann. Inst. H. Poincar{\'e} Anal. Non
Lin{\'e}aire 18 (2001), 687-760.}

\bibitem{calions}{T. Cazenave and P. L. Lions, }{Orbital stability of
standing waves for some nonlinear Schr\"{o}dinger equations,
}{Comm. Math. Phys. 85 (1982), 549-561.}


\bibitem{mugnai}{T. D'Aprile and D. Mugnai, }{Solitary waves for nonlinear Klein-Gordon-Maxwell and
Schr{\"o}dinger-Maxwell equations, }{Proc. Royal Soc. Edinburgh
134A (2004), 893-906.}

\bibitem{mugnai2}{T. D'Aprile and D. Mugnai, }{Non-existence results for the
coupled Klein-Gordon-Maxwell equations, }{Adv. in Nonl. Studies 4
(2004), 307-322.}

\bibitem{daprile2}{T. D'Aprile and J. Wei, }{On bound states concentrating on spheres for the
Maxwell-Schr{\"o}dinger equation, }{SIAM J. Math. Anal. 37 (2005),
321-342.}

\bibitem{daprile3}{T. D'Aprile and J. Wei, }{Standing waves in the Maxwell-Schr{\"o}dinger equation
and an optimal configuration problem, }{Calc. Var. 25 (2005),
105-137.}

\bibitem{gnn}{B. Gidas, W.-M. Ni and L. Nirenberg, }
{Symmetry and related properties via the maximum principle,
}{Comm. Math. Phys. 68 (1979), 209-243.}

\bibitem{grillakis}{M. Grillakis, J. Shatah and W. Strauss, }{ Stability
theory of solitary waves in the presence of symmetry I y II, }{J.
Funct. Anal. 74 (1987), 160-197 y 94 (1990), 308-348.}

\bibitem{ianni}{I. Ianni and G. Vaira, }{Semiclassical states for the Schr{\"o}dinger-Poisson
problem with an external potential and a density charge:
concentration around a sphere, }{preprint.}

\bibitem{kikuchi}{H. Kikuchi, }
{On the existence of a solution for elliptic system related to the
Maxwell-Schr{\"o}dinger equations, }{Nonlinear Anal. 67 (2007), no. 5,
1445-1456.}

\bibitem{kikutesis}{H. Kikuchi, }{Existence and orbital stability of standing waves for nonlinear
Schr{\"o}dinger equations via the variational method, }{Doctoral
Thesis.}

\bibitem{lieb}{E. H. Lieb and M. Loss, }{Analysis, }{Graduate
Studies in Mathematics, vol. 14, AMS, 1997.}

\bibitem{lions}{P.-L. Lions, }{Solutions of Hartree-Fock
equations for Coulomb systems, }{Comm. Math. Physics 109 (1984),
33-97.}


\bibitem{mauser}{N.J. Mauser, The Schr\"odinger-Poisson-X$\a$ equation, Applied Math. Letters 14 (2001),
759-763.}

\bibitem{pisani}{L. Pisani and G. Siciliano, }{
Neumann condition in the Schr{\"o}dinger-Maxwell system, }{Topol.
Methods Nonlinear Anal. 29 (2007), 251-264.}


\bibitem{M3AS}{D. Ruiz, }{Semiclassical states for coupled Schrodinger-Maxwell equations: concentration around a
sphere}{, M$^3$AS 15 (2005), 141-164.}

\bibitem{JFA}{D. Ruiz, }{The Schr{\"o}dinger-Poisson equation under the effect of a nonlinear local
term, }{J. Funct. Anal. 237 (2006), 655-674.}


\bibitem{oscar}{O. S{\'a}nchez and J. Soler, }{Long-time dynamics of the Schr{\"o}dinger-Poisson-Slater system, }
{J. Statistical Physics 114 (2004), 179-204.}

\bibitem{slater}{J.C. Slater, }{A simplification of the Hartree-Fock method, }
{Phys. Review 81 (1951), 385-390.}

\bibitem{smets}{D. Smets, J. Su and M. Willem, }{
Non-radial ground states for the H{\'e}non equation, }{Commun.
Contemp. Math. 4 (2002), 467-480.}


\bibitem{swwillem}{J. Su, Z.-Q. Wang and M. Willem, }{
 Nonlinear Schr\"{o}dinger equations with unbounded and decaying radial
 potentials, }{Commun. Contemp. Math. 9 (2007), 571-583.}

\bibitem{swwillem2}{J. Su, Z.-Q. Wang and M. Willem, }{Weighted Sobolev
embedding with unbounded and decaying radial potentials,}{ J.
Differential Equations 238 (2007), 201-219. }

\bibitem{zhou}{Zhengping Wang and Huan-Song
Zhou, }{Positive solution for a nonlinear stationary
Schr\"{o}dinger-Poisson system in $\R^3$, }{Discrete and
Continuous Dynamical Systems, 18 (2007), 809-816.}

 \end{thebibliography}
\end{document}